\documentclass[12pt]{article}

\usepackage{amssymb,amsthm,amsmath}
\usepackage[usenames,dvipsnames]{xcolor}
\allowdisplaybreaks
\usepackage{color,graphicx,epsfig}
\usepackage{mathrsfs}

\newcommand{\comment}[1]{}
\definecolor{teal}{RGB}{0,128,128}
\definecolor{darkpurple}{RGB}{128,0,128}

\usepackage[titletoc,toc,title]{appendix}

\newcommand{\tmod}[1]{{\;\rm (mod\; #1)}}

\newcommand{\ol}{\overline}

\newtheorem{theorem}{Theorem}[section]

\newtheorem{lemma}[theorem]{Lemma}
\newtheorem{cor}[theorem]{Corollary}

\theoremstyle{definition}

\newtheorem{rem}[theorem]{Remark}
\theoremstyle{definition}

\def \cG {{\cal G}}

\def \cR {{\cal R}}

\def \Z {\mathbb Z}

\title {
Constructing uniform 2-factorizations via row-sum matrices: solutions to the Hamilton-Waterloo problem\\
}
\author{
A.\ C.\ Burgess \footnotemark[1]  
\and 
P.\ Danziger    \footnotemark[2] 
\and
A.\ Pastine     \footnotemark[3] 
\and
T.\ Traetta     \footnotemark[4] 
}

\begin{document}
\maketitle

\footnotetext[1]{Department of Mathematics and Statistics, University of New Brunswick, 100 Tucker Park Rd., Saint John, NB  E2L 4L5, Canada. Email: andrea.burgess@unb.ca}
\footnotetext[2]{Department of Mathematics, Toronto Metropolitan University, 350 Victoria St., Toronto, ON  M5B 2K3, Canada. Email: danziger@ryerson.ca}
\footnotetext[4]{DICATAM, Universit\`{a} degli Studi di Brescia, Via Branze 43, 25123 Brescia, Italy. E-mail: tommaso.traetta@unibs.it}
\footnotetext[3]{Instituto de Matem\'atica Aplicada San Luis (UNSL-CONICET),  Universidad Nacional de San Luis, San Luis, Argentina. E-mail: agpastine@unsl.edu.ar}
\comment{
\address[AB]{Department of Mathematics and Statistics, University of New Brunswick, 100 Tucker Park Rd., Saint John, NB  E2L 4L5, Canada.}
\address[PD]{Department of Mathematics, Toronto Metropolitan University, 350 Victoria St., Toronto, ON  M5B 2K3, Canada. }
\address[AP]{Instituto de Matem\'{a}tica Aplicadas de San Luis, CONICET-UNSL.}
\address[TT]{DICATAM, Universit\`{a} degli Studi di Brescia, Via Branze 43, 25123 Brescia, Italy.}

\author[AB]{A.\ C.\ Burgess}\ead{andrea.burgess@unb.ca} 
\author[PD]{P.\ Danziger}\ead{danziger@ryerson.ca} 
\author[AP]{A.\ Pastine}\ead{agpastine@unsl.edu.ar} 
\author[TT]{T.\ Traetta}\ead{tommaso.traetta@unibs.it} 
}

\begin{abstract}
In this paper, we formally introduce the concept of a row-sum matrix over an arbitrary group $G$. When $G$ is cyclic, these types of matrices have been widely used to build uniform 2-factorizations of small Cayley graphs (or, Cayley subgraphs of blown-up cycles), which themselves  factorize complete (equipartite) graphs. 

Here, we construct row-sum matrices over a class of non-abelian groups, the generalized dihedral groups, and we use them to construct uniform $2$-factorizations that solve infinitely many open cases of the Hamilton-Waterloo problem, thus filling up large parts of the gaps in the spectrum of orders for which such factorizations are known to exist.
\end{abstract}

\comment{
}	

\section{Introduction}
In this paper we denote by $L=[^{\alpha_1} a_1,  \ldots, \;^{\alpha_\ell} a_\ell]$ the 
multiset containing $\alpha_i\geq0$ copies of the element $a_i\in A$, for each $i\in\{1,\ldots, \ell\}$. Note that the $a_i$s  need not be distinct. We will call such a multiset a {\em list}, even though order is not important, as we will be dealing extensively with so called $v$-lists, which are multisets as above where $\alpha_1+\alpha_2+\cdots+\alpha_\ell = v$ (see Section 2). 

Given a simple graph $G$, we denote by $V(G)$ and $E(G)$ its sets of vertices and edges, respectively. 
As usual, we denote by $C_\ell$ a \emph{cycle} of length $\ell$ (briefly, an $\ell$-cycle), and by
$(x_0,x_1,\ldots,x_{\ell-1})$ the $\ell$-cycle with edges $x_0x_1,x_1x_2,\ldots,$ $x_{\ell-1} x_0$.
A {\em factor} of $G$ is a spanning subgraph $F$ of $G$; when $F$ is $i$-regular, we speak of an $i$-factor.
In particular, a 1-factor (resp.\ a $2$-factor) of $G$ is a vertex-disjoint union of edges (cycles) whose vertices cover $V(G)$.  A $2$-factor $F$ of $G$ containing only cycles of length $\ell$ will be called a {\em $C_\ell$-factor} or {\em uniform factor}.

By $K_v^*$ we mean the \emph{complete graph} $K_v$ on $v$ vertices when $v$ is odd and $K_v-I$, that is, $K_v$ minus the edges of the $1$-factor $I$, when $v$ is even. Also, by $K_t[z]$ we denote the \emph{complete equipartite graph} with $t$ parts of size $z\geq 1$. Note that $K_t[1] = K_t$.

A \emph{2-factorization} of a simple graph $G$ is a set $\cG$ of $2$-factors of $G$ whose edge sets
partition $E(G)$. It is well known that $G$ has a 2-factorization if and only if it is regular of even degree.
However, if we require the factors of $\cG$ to have a specific structure then the problem becomes much harder.
For example, the existence of a $2$-factorization of $G$ into copies of a given $2$-factor $F$ is an open problem even when $G=K^*_v$. This is the well-known {\em Oberwolfach Problem}, originally posed by Ringel in 1967 for odd $v$.
A survey of the most relevant results on this problem, updated to 2006, can be found in 
\cite[Section VI.12]{Handbook}. For more recent results we refer the reader to \cite{BDT22}.

A factorization of the simple graph $G$ into copies of a $C_\ell$-factor is briefly called a {\em $C_\ell$-factorization} or {\em uniform factorization} of $G$. The problem of factoring $K_v^*$ into copies of a uniform 2-factor, that is, the uniform Oberwolfach Problem, has been solved  \cite{AH85, ASSW, HS91, RW71}.
\begin{theorem}[\cite{AH85, ASSW, HS91, RW71}]
\label{OP uniform}
Let $v, \ell \geq 3$ be integers.
There is a $C_\ell$-factorization of $K_v^*$ if and only if $\ell \mid v$, except that there is no $C_3$-factorization of $K_{6}^*$ or $K_{12}^*$.
\end{theorem}

A similar result when $G=K_t[z]$ has more recently been obtained in \cite{Liu00, Liu03}.

\begin{theorem}[\cite{Liu00, Liu03}]
\label{Liu} 
Let $\ell, t$ and $z$ be positive integers with $\ell\geq 3$.
There exists a $C_\ell$-factorization of $K_t[z]$ if and only if $\ell\mid tz$, $(t-1)z$ is even, 
further $\ell$ is even when $t = 2$, and
$(\ell, t, z) \not\in \{(3, 3, 2), (3, 6, 2), (3, 3, 6), (6, 2, 6)\}$.
\end{theorem}

We may generalise this problem to the {\em Generalized Oberwolfach Problem}, denoted  GOP$(G; \cR)$, where $\cR = \{^{\alpha_1}R_1, \ldots, \;^{\alpha_t}R_t\}$ is a list of $2$-factors of $G$, where each $R_i$ is repeated $\alpha_i$ times (with $\alpha_i$ a positive integer) and the $R_i$ are pairwise non-isomorphic.
The Generalized Oberwolfach Problem then requires that the edges of $G$ be factored into a union of $\alpha_i$ copies of $R_i$, $1\leq i\leq t$. 
If each $R_i$ is uniform, with cycles of length $a_i$, we speak of GOP$(G; [^{\alpha_1}a_1,\ldots, ^{\alpha_t}a_t])$.
Since the $R_i$ are factors and every edge of $G$ is in one of the factors $R_i$, this requires that $G$ is regular with each vertex having degree $2\sum_{i=1}^t \alpha_i$ and that if $R_i$ is a $C_{a_i}$-factor, then $a_i$ divides the order of $G$. 
Despite recent probabilistic results which show eventual existence, these results are non-constructive and give no lower bounds for their implementation and so this problem remains wide open; see \cite{BDT19} for more details on the Generalized Oberwolfach Problem.

When $\cR=\{^{\alpha}R_1, \;^{\beta}R_2\}$, then GOP$(G; \cR)$ represents the most studied variant of the Oberwolfach Problem, known as the {\em Hamilton-Wa\-ter\-loo Problem}, and denoted by  HWP$(G;R_1,R_2;\alpha,\beta)$, or HWP$(v;R_1,R_2;\alpha,\beta)$ when $G$ is $K_v^*$. This problem asks for a factorization of $G$ into 
$\alpha$ copies of $R_1$ and $\beta$ copies of $R_2$.
In the case where $R_1$ and $R_2$ are a $C_M$-factor and $C_N$-factor, respectively, we refer to 
HWP$(G;M,N;\alpha,\beta)$, or HWP$(v;M,N;\alpha,\beta)$ when $G$ is $K_v^*$, and speak of the uniform Hamilton-Waterloo Problem.
Clearly, when $\alpha = 0$ or $\beta=0$ we obtain the uniform Oberwolfach problem which is completely solved (Theorem \ref{OP uniform}). Therefore, from now on we will assume that both $\alpha$ and $\beta$ are positive integers.
Well-known obvious necessary conditions for the solvability of HWP$(G;M,N;\alpha,\beta)$ are given the following theorem.

\begin{theorem}
\label{nec} Let $G$ be a graph of order $v$, and let 
$M,N, \alpha$ and $\beta$ be  non-negative integers.
In order for a solution of HWP$(G; M,N; \alpha,\beta)$ to exist,
$M$ and $N$ must be divisors of $v$ greater than $2$, and $G$ must be regular of degree $2(\alpha+\beta)$. 
\end{theorem}

We are interested in constructing solutions to the uniform Hamilton-Waterloo Problem. 
We point out that this case (as well as the general problem) is still open, and this is quite surprising considering that the equivalent problem of factoring $K_v^*$ into uniform factors (the uniform OP) was solved in the nineties (see Theorem \ref{OP uniform}).

For more details and some history on the problem, we refer the reader to \cite{BDT18}. 
That paper deals with the case where both $M$ and $N$ are odd positive integers and provides an almost complete solution to HWP$(v;M,N;\alpha,\beta)$ for odd $v$.
If $M$ and $N$ are both even, then HWP$(v;M,N;\alpha,\beta)$ has a solution except possibly when $\alpha=1$ or 
$\beta=1$; $\beta=3$, $v \equiv 2$ (mod $4$) and $\gcd(M,N)=2$; or $v=MN/\gcd(M,N)\equiv 2$ (mod $4$)  \cite{BryantDanziger, BDT19b}.  However, the problem is completely solved when $M$ and $N$ are even and $M$ is a divisor of $N$ \cite{BryantDanzigerDean}. The case where $M$ and $N$ have different parities is the most challenging. 
Indeed, the only case where $M\not\equiv N \tmod{2}$ that
has been completely solved is when $(M,N)=(3,4)$ 
\cite{BonviciniBuratti,  DanzigerQuattrocchiStevens, OdabasiOzkan, WangChenCao}. 
The only other cases which have been considered are when
$M$ is a divisor of $N$ \cite{AsplundEtAl, BDT18b,LeiShen}; 
$M=4$ \cite{KeranenOzkan, OdabasiOzkan}; $M=8$ \cite{WangCao}; and when $M$ and $N$ are not coprime, $M$ is odd, $N=2^kn$ and $4^k$ divides $v$ \cite{KPdiffpar}.
However, possible exceptions remain in all of these cases.
The following theorem summarizes the results in \cite{BryantDanziger,BryantDanzigerDean,BDT18,BDT18b,BDT19b, KPdiffpar}.
\begin{theorem}[\cite{BryantDanziger,BryantDanzigerDean,BDT18,BDT18b,BDT19b,KPdiffpar}]
There is a solution to HWP$(v;M,N;\alpha,\beta)$ when 
\begin{enumerate}
\item $M,N\geq 3$ are odd, $M\geq N$, $MN/\gcd(M,N)$ divides  $v$
and $v\neq MN/\gcd(M,N)$, except possibly if $\beta\in\{1,3\}$.
\item $M$ and $N$ are even, with $M>N$, and $M$ and $N$ divide $v$, except possibly if $N$ does not divide $M$, and $1 \in \{\alpha,\beta\}$; if $\beta=3$, $v\equiv 2 \pmod{4}$ and $\gcd(M,N)=2$; or if $v=MN/\gcd(M,N) \equiv 2\pmod{4}$, and $\alpha$ and $\beta$ are odd.
\item $N=2^kn$ with $k\geq 1$, $M$ and $n$ are odd, and either $M$ divides $n$, $v>6N>36M$ and $s\geq 3$; or 
$\gcd(M,n)\geq 3$, $4^k$ divides $v$, $v/(4^k\,\mathrm{lcm}(M,n))$ is at least $3$ and $1\not\in\{\alpha,\beta\}$.
\end{enumerate}
\end{theorem}
The results in \cite{BDT18,BDT18b,KPdiffpar} were obtained using solutions to HWP$(C_g[u],M,N,$ $\alpha,\beta)$,
where $C_g[u]$ is the graph obtained from $C_g$ by replacing every vertex in the cycle with $u$ copies of it.
In other words, it is the graph with vertices of the form $x_{i,j}$, $0\leq i \leq g-1$, $0\leq j \leq u-1$,
and edges of the form $x_{i,a}x_{i+1,b}$ with addition done modulo $u$, and $0\leq a,b \leq u-1$.

In this paper, we make further progress when $M$ and $N$ are not coprime in two regards. On
the one hand, we improve the result for the case when $M$ and $N$ have different parities,
changing the condition that $4^k$ divides $v$ into the condition that $2^{k+2}$ divides
$v$. On the other hand, our results put no restrictions on $\alpha$ and $\beta$, covering
the difficult case when $M$ and $N$ have the same parity and $1\in \{\alpha,\beta\}$
that was previously left open. More precisely, our main result is the following.
\begin{theorem}\label{main}
Let $v$, $M$ and $N$ be integers greater than 3, and let $\ell=\mbox{lcm}(M,N)$.
A solution to $\mathrm{HWP}(v; M, N; \alpha, \beta)$ exists if and only if $M\mid v$ and $N\mid v$, except possibly when
\begin{itemize}
\item $\gcd(M,N)\in \{1,2\}$;
\item $4$ does not divide $v/\ell$;
\item $v/4\ell \in\{1,2\}$;
\item $v=16\ell$ and $\gcd(M,N)$ is odd;
\item $v=24\ell$ and $\gcd(M,N)=3$.
\end{itemize}
\end{theorem}

In Section \ref{sec:prelim} we introduce the concept of row-sum matrices and prove some preliminary results.
In particular, we show
how to use said matrices to obtain solutions to HWP$(C_g[u],M,N,\alpha,\beta)$. Next, in Section \ref{RSM_dihedral}
we prove the existence of the matrices we need for our main result. Finally, in Section \ref{sec:main}
we complete the proof of Theorem \ref{main}.

\section{Preliminary results}\label{sec:prelim}

Recall that given a group $\Gamma$, and an integer $v$,
a $v$-list of $\Gamma$ is a list $\Delta = [\delta_1, \ldots, \delta_v]$
of $v$ (not necessarily distinct) elements of $\Gamma$.
Given an integer $g$, set $g\Delta =  [g\delta_1, \ldots, g\delta_v]$. 
It will be helpful to refer to the list $\omega(\Delta)$ of element orders associated to $\Delta$, defined as follows:
$\omega(\Delta) = [\omega(\delta_1), \ldots, \omega(\delta_v)]$, 
where each $\omega(\delta_i)$ is the order of the element $\delta_i$ of the group $\Gamma$, and hence a divisor of the order of $\Gamma$.

\subsection{$\Delta$-Permutations}
Let $\Gamma$ be an arbitrary group of order $v$, 
and let $\Delta$ be a $v$-list of $\Gamma$.
We say that a permutation $\varphi$ of $\Gamma$ is a $\Delta$-permutation if the following condition holds:
\begin{equation}\label{deltapermutation}
\left[\varphi(a_1)-a_1, \varphi(a_2)-a_2, \ldots, \varphi(a_v)-a_v\right] = \Delta,
\end{equation}
where $\{a_1,a_2,\ldots, a_v\}=\Gamma$.

\begin{rem}\label{rem:deltapermutation}
Given any fixed $x\in \Gamma$ and $\delta\in\Delta$, we can assume that $\varphi(x)-x=\delta$. Otherwise,
take $a_j\in \Gamma$ such that $\varphi(a_j)-a_j=\delta$, 
and set $b_i=a_i-a_j+x$ and $\phi(b_i) = \varphi(a_i)-a_j+x$, for every $i=1,\ldots,v$.
Clearly, $\Gamma=\{b_1, b_2, \ldots, b_v\}$ and $\phi$ is a $\Delta$-permutation of $\Gamma$, since 
$\phi(b_i) - b_i = \varphi(a_i)-a_i$ for each $i=1,\ldots, v$.
Also, $b_j=x$ and $\phi(x)-x=\varphi(a_j)-a_j=\delta$.
\end{rem}

Given an arbitrary $v$-list $\Delta$ of an abelian group $G$ of order $v$
a necessary condition for a $\Delta$-permutation of $G$ to exist is that
$\sum\Delta=0$, where $\sum\Delta$ denotes the sum of the elements in $\Delta$.
M. Hall \cite{Hall} proved that this condition is also sufficient.

\begin{theorem}[\cite{Hall}]\label{teohall}
Let $G$ be an abelian group of order $v$, and let $\Delta$ be a $v$-list of $G$. 
There exists a $\Delta$-permutation of $G$ if and only if
$\sum \Delta=0$. 
\end{theorem}

The following special $\Delta$-permutations will be useful in the constructions of Section 4.

\begin{theorem}\label{special_1}
Let $m$ and $n$ be positive integers with $m$ odd. Then
there exists a $\Delta$-permutation $\psi$ of $\Z_m\times\Z_{2n}$ such that
\begin{enumerate}
\item $\Delta=
\left[ \,^1(0,0)\right] \, \cup \, 
\left[^1\gamma\mid \gamma\in \Z_m\times\Z_{2n}, \gamma\neq (0,n)\right],$
\item $\psi$ fixes $(0,0)$ and 
$\left(-\frac{m-1}{2}, \left\lfloor\frac{n+1}{2}\right\rfloor + \frac{m-1}{2}n\right)$.
\end{enumerate}
\end{theorem}
\begin{proof}  Assume that $V(K_{2mn}) = \Z_m\times\Z_{2n}$.
We are going to construct a suitable matching $H$ of $K_{2mn}$ with $mn-1$ edges.
We leave to the reader the check that the permutation that swaps every pair of adjacent vertices in $H$ and fixes the remaining two is the desired $\Delta$-permutation of $\Z_m\times\Z_{2n}$.

We consider the following matching of $K_{2n}$,
\[\textstyle{
  F=\left\{   
          \{j,-j\} \,\Big|\, j = 1, \ldots, \left\lfloor \frac{n-1}{2}\right\rfloor
      \right\}
      \,\cup\,
      \left\{ 
          \{j,-j+1\} \,\Big|\, j = \left\lfloor \frac{n+3}{2}\right\rfloor, \ldots, n
      \right\},
      }
\]
and note that $V(F) = \Z_{2n}\setminus\{0, u\}$ with 
$u=\left\lfloor\frac{n+1}{2}\right\rfloor$. 

For every non-negative integer $i$, we lift $F$ to a matching $F(i)$ with vertex-set 
$\{\pm i\}\times (\Z_{2n}\setminus\{0,u\})$ defined as follows:
\[
  F(i)=\big\{\{(x, y_1), (-x, y_2)\} \,\big|\, x=\pm i, 
  \{y_1, y_2\}\in E(F)\big\}.
\]
Consider also the following two matchings of $K_{2mn}$:
\begin{align*}
  F' &= 
\textstyle{
       \left\{
         \{(i,in), -(i,in)\} \,\big|\, i=1,\ldots, \frac{m-1}{2} 
       \right\},
}  
\\
F'' &= 
\textstyle{
       \left\{
         \{(-i+1,u+in+n), (i,u+in)\} \,\big|\, i=1,\ldots, \frac{m-1}{2} 
       \right\},  
}            
\end{align*}
and set
\[ H = \bigcup_{i=0}^{(m-1)/2} \big(F(i) + (0,in)\big) \,\cup\, F' \,\cup\, F''.
\]
It is not difficult to check that $H$ is a matching
of $K_{2mn}$ with $2mn-1$ edges, missing the vertices $(0,0)$
and $\left(-\frac{m-1}{2}, u+\frac{m-1}{2}n\right)$.
\end{proof}

\begin{theorem}\label{special_2}
Let $m\geq 1$ and $n\geq 3$ be odd integers. Then
there exists a $\Delta$-permutation $\psi$ of $\Z_m\times\Z_{2n}$ such that
\begin{enumerate}
  \item $
        \Delta=
          \left[ 
           \,^{2mn-6} (1,0), \,^{3} (2,0), \,^{1} (0,2),  \,^{1} (0,n-2), \,^{1} (0,n)
         \right] 
        $,        
  \item $\psi(0,0)=(0,n), \psi(0,n) = (0,n+2)$.
\end{enumerate}  
\end{theorem}
\begin{proof} Let $\psi$ be the permutation of $\Z_m\times\Z_{2n}$ defined as follows
\[ 
 \psi(0,0)=(0,n),\;\;\; \psi(0,n) = (0,n+2), \;\;\; \psi(0,n+2) = (0,0), 
\]
and for every $z\in \Z_m\times\Z_{2n} \setminus\{(0,0), (0,n), (0,n+2)\}$, let
\[
  \psi(z)=
  \begin{cases}
    z+(2,0) & \text{if $z\in \{(-1,0), (-1,n), (-1,n+2)\}$} ,\\
    z+(1,0) & \text{otherwise}.
  \end{cases}
\]
One can check that $\psi$ is the desired permutation.
\end{proof}

\subsection{Row-sum matrices and $2$-factorizations of $C_g[n]$}

Let $\Gamma$ be a group, and let $S \subset \Gamma$. Also, let $\Sigma$ be an $|S|$-list of elements of $\Gamma$.
A \emph{row-sum matrix} $RSM_\Gamma(S, g; \Sigma)$ is an
$|S|\times g$ matrix, whose $g\geq 2$ columns are permutations of $S$ and such that the list of (left-to-right) row-sums is $\Sigma$. We write $RSM_\Gamma(S, g; \omega(\Sigma))$ whenever we are just interested in the list $\omega(\Sigma)$ of orders of the row-sums.
Notice that an $RSM_\Gamma(\Gamma, 2; \Sigma)$ is equivalent to a $\Sigma$-permutation of $\Gamma$.

Row-sum matrices are useful to build factorizations of suitable Cayley subgraphs of $C_g[n]$. More precisely, we denote by
$C_g[\Gamma, S]$ ($g\geq 3$) the graph with point set 
$\mathbb{Z}_g\times \Gamma$ and edges $(i,x)(i+1,d+x)$, $i\in \mathbb{Z}_g$, $x\in\Gamma$ and $d\in S$. In other words, 
$C_g[\Gamma,S]= {\rm cay}(\Z_g\times \Gamma, \{1\}\times S)$; hence, it is $2|S|$-regular.
It is straightforward to see that if $\Gamma$ has order $n$, then $C_{g}[n] \cong C_{g}[\Gamma,\Gamma]$;  
hence, $C_{g}[\Gamma,S]$ is a subgraph of $C_g[n]$.

The following result, proven in \cite[Theorem 2.1]{{BDT19}} when $\Gamma$ has odd order, 
shows that row-sum matrices can be used to build factorizations of $C_g[\Gamma, S]$.

\begin{theorem}\label{row-sum matrix to factorizations}
If there exists an $RSM_\Gamma(S,g;\Sigma)$, then
$\mathrm{GOP}(C_{g}[\Gamma, S]; g\omega(\Sigma))$ has a solution.
\end{theorem}

We skip the proof of Theorem \ref{row-sum matrix to factorizations} when $|\Gamma|$ is even, as it is identical to the odd case.

We now show that row-sum matrices can be easily extended over the columns
whenever $S$ is closed under taking negatives (that is, $S=-S$) or when 
$S=\Gamma$ has a \emph{complete mapping}. We recall that a 
complete mapping of $\Gamma$ is a permutation $\pi$ 
of $\Gamma$ such that $\rho(x)=x+\pi(x)$ is also a permutation.

\begin{theorem}\label{row-sum matrix extension}
If there is an $RSM_\Gamma(S,g;\Sigma)$, then there exists
$RSM_\Gamma(S,g+i;\Sigma)$ in each of the following cases:
\begin{enumerate}
  \item $S=-S$ and $i\geq2$ is even, or
  \item $S=\Gamma$ has a complete mapping and $i\geq1$.
\end{enumerate}
\end{theorem}
\begin{proof} 
  Let $A$ be an $RSM_\Gamma(S,g;\Sigma)$.
  We first assume that $S=-S$ and $i\geq2$ is even. 
  The existence of an $RSM_\Gamma(S,g+i;\Sigma)$ was essentially proven in \cite[Theorem 2.1]{{BDT19}} as follows: it is enough to extend $A$ by adding $i/2$ copies of an $RSM_\Gamma(S,2; \Sigma')$, say $B$, where
  $\Sigma'$ is the $|S|$-list of zeros (i.e., each row of $B$ sums to $0$). Clearly, $B$ can be easily built by choosing its second column to be the negative 
  of the first one, which must be a permutation of $S$ by assumption.
  
  Now let $i\geq1$, assume that $S=\Gamma$ has a complete mapping $\pi$, and set $\rho(x)=x+\pi(x)$. In view of the first part of the proof, it suffices to show the existence of an $RSM_\Gamma(S,g+1;\Sigma)$. 
To do so, it is enough to replace each element of the last column of $A$, say $y$, with the pair $(x, \pi(x))$ where $x=\rho^{-1}(y)$.
\end{proof}

Throughout the paper we will only deal with solvable groups, and we will make use of a result by Hall and Paige who proved that a finite solvable group $\Gamma$ has a complete mapping if and only if its $2$-Sylow subgroup is either trivial or non-cyclic. This result was then extended to an arbitray group (see \cite{Ev18}).

\begin{theorem}[\cite{Hall-Paige}]\label{teo:completemappings}
A finite solvable group $\Gamma$ has a complete mapping if and only if 
its $2$-Sylow subgroup is either trivial or non-cyclic.
\end{theorem}

This, combined with
Theorems \ref{row-sum matrix to factorizations} and \ref{row-sum matrix extension} means that when $S=\Gamma$ is as in Theorem \ref{teo:completemappings}, it is enough to construct row-sum matrices with $3$ columns.

\begin{cor}\label{lem:g=3impliesall}
Let $\Gamma$ be a solvable group of order $n$  whose $2$-Sylow subgroup is either trivial or non-cyclic. If there is an $RSM_\Gamma(\Gamma,3;\Sigma)$, then there exist
\begin{enumerate}
\item an $RSM_\Gamma(\Gamma,g;\Sigma)$, and
\item a solution to $GOP(C_{g}[n]; g\omega(\Sigma))$,
\end{enumerate}
for every $g\geq 3$.
\end{cor}

We end this section by constructing row-sum matrices over an abelian group having a given list of associated row-sum orders.
\begin{theorem}\label{ell}
Let $\Gamma=\Z_2\times \Z_{2^\ell}\times \Z_{mn}$ with $n$ odd and $\ell\geq 1$.
Then, there is an $RSM_\Gamma(\Gamma,g;[^{2\gamma} m,   \;^{2\delta} 2^kn])$ whenever
$\gamma + \delta = 2^{\ell}mn$, $g\geq 3$ and 
$0\leq k\leq \ell$.
\end{theorem}
\begin{proof} 
Since $\ell\geq 1$, the Sylow $2$-subgroup of $\Gamma$ is non-cyclic. Hence, by Corollary \ref{lem:g=3impliesall}, it is enough to prove the assertion for $g=3$.
Set
\[
\Delta = 
         \left[  
           ^{\gamma} (0,0,n), \,^{\gamma}(0,0,-n), 
           \,^{\delta} (0,  2^{\ell-k},  m), 
           \,^{\delta} (0, -2^{\ell-k}, -m)
         \right]. 
\]
Since $\sum \Delta = 0 = \sum \Gamma$, by Theorem 
\ref{teohall}, there are a $\Delta$-permutation $\varphi$ of $\Gamma$
and a $\Gamma$-permutation $\psi$ of $\Gamma$.
Now consider the $2^{\ell+1}mn\times r$ matrix $A$ whose rows, indexed over $\Gamma$, have the following form 
\[
A_x = 
\big(-\psi(x)\quad x\quad \varphi(\psi(x)-x)\big).    
\]
One can easily check that each column of $A$ is a permutation of $\Gamma$, and 
\[
\big[\textstyle{\sum A_x} \;\big|\; x\in \Gamma \big]= \Delta.
\]
Since $\Delta$ contains $2\gamma$ elements of order $m$ and $2\delta$ elements of order $2^kn$, the assertion follows.
\end{proof}

\subsection{Matrices over a generalized dihedral group}\label{Sec:matrices}
From now on, 
$g,m$ and $n$ will denote positive integers with 
$g\geq 3$, $m$ and $n$ odd, and $G = \Z_{m} \times\Z_{2^{k+1}n}$ with $k\geq2$. Notice that we allow both $m$ and 
$n$ to be equal to $1$. In the case $m=1$ this means that we work with   $\Z_{1} \times\Z_{2^{k+1}n}\simeq  \Z_{2^{k+1}n}$.
Recall that the coordinatewise multiplication  by any element $x\in G$ 
is a homomorphism of the group $G$; in particular, the multiplication by 
$\epsilon=(-1,-1)$ is a group automorphism, and its order is $2$ since $\epsilon^2=(1,1)$. 
Therefore, we can define the semidirect product 
$G \rtimes \Z_2$ whose underlying set is 
$G\times \Z_2$, 
and the group operation, still denoted by $+$, is defined as follows:
\[
  (x, \tau) + (x', \tau') = (x+ \epsilon^\tau x', \tau+\tau').
\]
It is not difficult to check that $G \rtimes \Z_2 \simeq  Dih(\Z_{m} \times \Z_{2^{k+1}n})$
is the generalized dihedral group over $\Z_{m} \times \Z_{2^{k+1}n}$.

From now on, $\Gamma = G \rtimes \Z_2$ and for every subset $S\subseteq \Gamma$, we
simply write $C_g[S]$ in place of $C_g[\Gamma, S]$.
Note that $2\Gamma \simeq 2G = \Z_m \times 2\Z_{2^{k+1}n}$ and $|2\Gamma| = 2^{k}mn$.
Consider the interval $I=\{0, \ldots,2^kn-1\}$ and the function $\rho:2\Z_{2^{k+1}n}\rightarrow I$
with $\rho(y)$ being defined by $2\rho(y)=y$. Similarly, for every $x\in \Z_m$ ($m$ odd) we denote by $x/2$ the unique element of $\Z_m$ such that $2(x/2)=x$.
Now let $\varphi = (\varphi_1, \varphi_2, \varphi_3)$, where 
each $\varphi_h$ is a  permutation of $\Z_m \times  2\Z_{2^{k+1}n}$, and define the following five bijections:
\begin{align*}
  a_h&: \Z_m \times  2\Z_{2^{k+1}n}   \rightarrow \Z_m \times (2\Z_{2^{k+1}n}+1),\; (x,y)\mapsto \varphi_h(x,y) + (0,1),
\end{align*}
for $h\in\{1,2,3\}$, and 
\begin{align*}  
    b&: \Z_m \times  2\Z_{2^{k+1}n}   \rightarrow  \Z_m \times  (-I),\; (x,y)\mapsto (-x/2,-\rho(y)),\\
    c&: \Z_m \times  2\Z_{2^{k+1}n}   \rightarrow  \Z_m \times (I+1),\; (x,y)\mapsto (x/2,\rho(y)+1).
\end{align*}
Finally, for every  $(x,y)\in \Z_m \times  2\Z_{2^{k+1}n}$, 
let $A(x,y)$ and $A'(x,y)$ be the $3\times3$ matrices with entries from $\Gamma$ defined as follows: 
\[
A(\varphi,(x,y)) =
 \left[\begin{array}{rrr}
 (a_1(x,y),0) \;&\;   (b(x,y), 1) \;&\;  (c(x,y), 1) \\[2mm]
  (c(x,y), 1) \;&\;  (a_2(x,y),0) \;&\;  (b(x,y), 1) \\[2mm]
  (b(x,y), 1) \;&\;   (c(x,y), 1) \;&\; (a_3(x,y),0) 
\end{array}\right]
\]
\[
A'(\varphi,(x,y)) =
 \left[\begin{array}{rrr}
 (a_1(x,y),0)  \;&\;  (c(x,y), 1) \;&\; (b(x,y), 1)\\[2mm]
  (c(x,y), 1)  \;&\;  (b(x,y), 1) \;&\; (a_2(x,y),0)\\[2mm]
  (b(x,y), 1)  \;&\; (a_3(x,y),0) \;&\; (c(x,y), 1)
\end{array}\right].
\]
Note that $A'(\varphi,(x,y))$ is obtained from $A(\varphi,(x,y))$ by swapping columns 2 and 3. 
From now on, given a matrix $M$ over an arbitrary group, 
$\sum M_h$ represents the (left-to-right) sum of the $h$-th row of $M$ denoted by $M_h$.

\begin{lemma}\label{A(x)rowsum}
For every $(x,y)\in\Z_{m}\times 2\Z_{2^{k+1}n}$, we have that 
\[
\sum A(\varphi,(x,y))_h =
\begin{cases}
 \big((\varphi_h(x,y)-(x,y)), 0\big)  & \text{if $h=1,3$}, \\
 \big((x,y)-\varphi_h(x,y)), 0\big)  & \text{if $h=2$}, \\  
\end{cases}
\]
\[
\sum A'(\varphi,(x,y))_h =
\begin{cases}
  \;\;\;\big(  (\varphi_h(x,y) + (x,y)),  0 \big) + (0,2,0)  & \text{if $h=1,2$}, \\
        -\big((\varphi_h(x,y) + (x,y)),  0 \big) - (0,2,0)  & \text{if $h=3$}. \\  
\end{cases}
\]
\end{lemma}
\begin{proof} Note that 
$\sum A(\varphi,(x,y))_h = (-1)^h(\sigma(x,y), 0)$, where
$\sigma = c-b-a_h$, for every $h=1,2,3$.
Also, 
\begin{align*}
\sigma(x,y) &= 
(x/2,\rho(y)+1)-(-x/2, -\rho(y))-(\varphi_h(x,y)+(0,1)) \\
&= (x, 2\rho(y))-\varphi_h(x,y) = (x,y)-\varphi_h(x,y).
\end{align*}
The result for $A'(\varphi,(x,y))$ is similar.
\end{proof}

\section{Row-sum matrices over a generalized dihedral group}
\label{RSM_dihedral}
In this section, we build row-sum matrices over the group $\Gamma \simeq Dih(\Z_{m} \times \Z_{2^{k+1}n})$
(defined in Section \ref{Sec:matrices}) and prove the following.

\begin{theorem}\label{row-sum exists}
Let $g\geq 3$, $k\geq 0$, and let $m,n\geq 1$ be odd integers. 
Then $RSM_\Gamma(\Gamma,g;[^{\alpha} m,   \;^{\beta} 2^kn])$ exists
if and only if $\alpha+\beta=2^{k+2}mn$ except possibly in the following cases 
\begin{enumerate}
  \item $k\geq 2$ and $\beta=0$, or
  \item $k=1$ and $\alpha,\beta\in\{0,2,4\}$.
\end{enumerate}
\end{theorem}

Its proof  is given in Sections 
\ref{case2}, \ref{case1} and \ref{case0} which respectively deal with the cases $k\geq 2$, $k=1$
and $k=0$. It mostly relies on Theorem \ref{ell}, and the following Theorems \ref{Gamma-2Gamma} and \ref{2Gamma}. 

Recall that $g,m,n$ are positive integers with $g\geq 3$, and $m$ and $n$ odd.

\begin{theorem}\label{Gamma-2Gamma}
$RSM_\Gamma(\Gamma\setminus2\Gamma,g;[^{\alpha} m,   \;^{\beta} 2^kn])$ exists when
$\alpha\neq1$ and $\beta\neq1$, with  
$\alpha+\beta = 3\cdot 2^kmn$.
\end{theorem}
\begin{proof} 
We first deal with the case where $\alpha$ and $\beta$ are even, with
$0\leq \alpha, \beta \leq 3\cdot 2^{k}mn$.
Letting $\alpha=2a$ and $\beta=2b$, and recalling that $|\Gamma\setminus2\Gamma| = 3\cdot 2^kmn$, 
we can write
\[\textstyle
  a = \sum_{h=1}^3 a_h \;\;\text{and}\;\; b = \sum_{h=1}^3 b_h,
\]
where $2(a_h + b_h) = 2^{k}mn = |2\Gamma|$,
for every $h\in\{1,2,3\}$.
By Theorem \ref{teohall}, there exists a $\Delta_h$-permutation 
$\varphi_h$ of $\Z_m\times 2\Z_{2^{k+1}n}$, where
\[\Delta_h=
    \left[ 
           \,^{a_h} (1,{0}), \,^{a_h} (-1,{0}), 
           \,^{b_h} (0,{2}), \,^{b_h} (0,-{2})
    \right].
\]
Notice that $|\Delta_h|=[\,^{2a_{h}}m, \,^{2b_{h}}2^kn]$, for every $h\in\{1,2,3\}$.
Let $A = \left[A(\varphi, (x,y))\right]$ be the column block-matrix whose blocks are 
the matrices $A(\varphi, (x,y))$ for $(x,y)\in \Z_m\times 2 \Z_{2^{k+1}n}$. Note that $A$ is a
$(3\cdot 2^{k}mn)\times 3$ matrix whose columns are permutations of $\Gamma\setminus 2\Gamma$.
Also, letting $L_A$ be the list of row-sums of $A$, 
by Lemma \ref{A(x)rowsum} we have that 
\[
\omega(L_A) = \omega(\Delta_1) \ \cup\ \omega(\Delta_2) \ \cup\ \omega(\Delta_3) = 
\left[\,^{2a}m, \,^{2b}2^kn\right],
\]
and the result follows.

It is left to deal with the case where $\alpha$ and $\beta$ are odd,
with $3\leq \alpha,\beta \leq 3\cdot 2^{k}mn$.
Let $A$ be the matrix built above with 
$2a= \alpha+3\geq 6$ and $2b = \beta-3$. Since $2a\geq 6$,
we can take $a_1, a_2, a_3\geq 1$. 
Also, by Remark \ref{rem:deltapermutation}, we can assume that
for $z = (\frac{m-1}{2}, 0)$ we have
$\varphi_h(z) =z+(1,0)$,   for every $h=1,2,3$.
We denote by $A'$ the 
$(3\cdot 2^{k}mn)\times 3$
matrix that we obtain 
by replacing the block 
$A(\varphi,z)$ of $A$ with $A'(\varphi,z)$. 
Clearly, 
the columns of $A'$ are still permutations of $\Gamma\setminus2\Gamma$. 
Also, note that by Lemma \ref{A(x)rowsum}, 
each $\sum  A(\varphi,z)_h$ has order $m$, 
whereas $\sum A'(\varphi,z)_h = \pm  (0,2,0)$ has order $2^kn$, for every $h=1,2,3$.
Therefore, denoting by $L_{A'}$ the list of row sums of $A'$,
we have that 
\[
\omega(L_{A'}) = [\,^{2a-3}m, \,^{2b+3}2^kn] = [\,^{\alpha}m, \,^{\beta}2^kn].
\]
The result then follows by applying Corollary \ref{lem:g=3impliesall} to $A'$.
\end{proof}

\begin{theorem}\label{2Gamma}
$RSM_\Gamma(2\Gamma,g;[^{\alpha} m,   \;^{\beta} 2^kn])$ 
exists whenever
$\alpha+\beta = 2^kmn$ and one of the following conditions holds:
\begin{enumerate}
  \item $k=0$, $\alpha\neq 1$ and $\beta\neq 1$, or
  \item $k=1$ and $\alpha,\beta\geq 3$ are  odd, or
  \item $k\geq 2$, $\alpha,\beta$ are even and $\beta\geq 2$.
\end{enumerate}
\end{theorem}
\begin{proof} Recall that $2\Gamma= 2G \times \{0\}$ and $2G = \Z_{m}\times 2\Z_{2^{k+1}n}$. 
Let 
\[
(\alpha, \beta)=
\begin{cases}
(2a+3, 2b+3) & \text{if $k=0,1$, and $\alpha,\beta$ are odd,} \\
(2a, 2b) & \text{if $k\neq1$, and $\alpha,\beta$ are even.}
\end{cases}
\]
By Theorem \ref{teohall}
and Remark \ref{rem:deltapermutation}, there are $\Lambda_i$-permutations $\psi_i$ of $2G$, with $i\in\{1,2\}$, such that
\begin{enumerate}
  \item $\Lambda_1= 
        \begin{cases}
        	      2G & \text{if $k=0$},\\
          \left[ \,^1(0,0)\right] \ \cup \
	      \left[^1z\mid z\in 2G, z\neq (0,2^{k}n)\right] & \text{if $k\geq 1$};
        \end{cases} 
        $			
  \item $
        \Lambda_2=
        \begin{cases}
          \left[ 
           \,^{1} (2,0), \,^{a} (1,0), \,^{a+2} (-1,0), \,^{1} (0,4), \,^{b} (0,2), \,^{b+2} (0,-2)
         \right] & \text{if $k=1$}, \\
         \left[ 
           \,^{a} (1,0), \,^{a} (-1,0), \,^{b} (0,2), \,^{b} (0,-2)
         \right]  & \text{if $k\geq2$};
        \end{cases} 
        $
  \item if $k\geq 1$, then $\psi_2(0,2^{k}n) = (0,2^{k}n) +
        \begin{cases}
           (0,4) & \text{if $k=1$},\\
           (0,-2) & \text{if $k\geq2$}.
        \end{cases}
        $
\end{enumerate}
Since $(0,0)\in\Lambda_1$, there exists a pair $\bar{z}\in 2G$ such that $\psi_1(\bar{z}) = \bar{z}$.
Let $B$ denote the $2^{k}mn \times 3$ matrix (with entries from $2\Gamma$) 
whose rows $B_z$, indexed over $2G$, are defined as follows: 
$B_z = \left[ (z,0) \;\;\; (-\psi_1(z),0) \;\;\; (\psi_2(w),0) \right]$ where
\[w=
  \begin{cases}
    (0, 2^{k}n)  & \text{if $z= \bar{z}$ and $k\neq 0$},\\
    \psi_1(z) - z & \text{otherwise}.
  \end{cases}
\]
Note that the columns of $B$ are permutations of $2\Gamma$. Also, 
one can check that for the list $L_B$ of row sums of $B$
we have 
\[
\omega(L_B) = [\,^{\alpha}m, \,^{\beta}2^kn].
\]
The result then follows by applying Lemma \ref{lem:g=3impliesall} to $B$.
\end{proof}

\subsection{The proof of Theorem \ref{row-sum exists} when $k\geq2$}
\label{case2}
Let $\alpha$ and $\beta$ be non-negative integers such that
$\alpha+\beta=2^{k+2}mn$. First, we assume that both $\alpha\neq1$ and $\beta\not\in\{0,1,3\}$ and let $\beta_1=2^kmn -\alpha_1$
$\beta_2=3\cdot 2^kmn -\alpha_2$ where $\alpha_1$ and $\alpha_2$ are defined as follows:
\begin{align*}
  \alpha_1=&
  \begin{cases}
    \min \left\{2^kmn-2, 
            2\left\lfloor 
             \frac{\alpha}{2} 
             \right\rfloor 
         \right\} & \text{if $\alpha\geq 2^kmn+1$,} \\      
    \alpha & \text{if $\alpha\leq 2^kmn$ is even,} \\   
    \alpha-3 & \text{if $3\leq \alpha< 2^kmn$ is odd,} \\       
  \end{cases}
\\
  \alpha_2 =& \alpha-\alpha_1.
\end{align*}
Clearly, $\alpha_1$ and $\beta_1$ are even, hence
Theorem \ref{2Gamma} guarantees the existence of  an $RSM_\Gamma(2\Gamma, g,[\,^{\alpha_1}m,\,^{\beta_1} 2^kn ])$, say $B$.
Furthermore, one can check that $\alpha_2\neq1$ and $\beta_2\neq 1$.
Hence, by Theorem \ref{Gamma-2Gamma} there is 
an $RSM_\Gamma(\Gamma\setminus 2\Gamma, g,[\,^{\alpha_2}m,\,^{\beta_2} 2^kn ])$, say $A$.
Therefore, 
$C=\left[
\begin{array}{c}
 A \\
 B
\end{array}
\right]
$ and Lemma \ref{lem:g=3impliesall} provide the desired RSM. 

It is  left to deal
with the cases $\alpha=1$ and $\beta=1,3$.

Case 1: $\alpha=1$.
Let 
$C = 
\left[\begin{array}{c}
  A \\
  B
\end{array}\right]
$ 
where $A$ is the matrix defined in Theorem \ref{Gamma-2Gamma}, with $a=0$ and $b=3\cdot 2^{k}mn$, 
and
$B$ is the matrix defined in Theorem \ref{2Gamma} with $a=0$ and $b=2^{k}mn$.
It follows that for the list $L_C$ of row sums of $C$ we have
\begin{equation}\label{Gamma:alpha=1}
\omega(L_C) = [\,^{2^{k+2}mn}2^kn].
\end{equation}
By Remark \ref{rem:deltapermutation}, 
we can assume that the permutations $\varphi_1$ and $\varphi_3$ used to define $A$ in Theorem 
\ref{Gamma-2Gamma} satisfy the condition $\varphi_h(z') = z'-(0,2) = (1,0)$, with $z'=(1,2)$, hence
\[a_h(z') = \varphi_h(z') + (0,1) = (1,1),\] 
for $h=1$ or $3$.
Furthermore, since $\Lambda_2 = 
            \left[\,^{2^{k-1}nm} (0,2), \,^{2^{k-1}nm} (0,-2) \right]$, we can assume that  
$\psi_2(0,0) = (0,2)$.
Consider the following matrices:
\[
  U=
  \left[\begin{array}{c}
  A(\varphi,z')_1 \\
  B_{(0,2)} \\
  A(\varphi,z')_3 \\    
\end{array}\right]
=
 \left[\begin{array}{rrr}
   (a_1(z'),0) \;&\;   (b(z'), 1)  \;&\;  (c(z'), 1)    \\[2mm]
      (0,2, 0) \;&\;    (0,-2, 0)  \;&\;  (0,2, 0)     \\[2mm]
    (b(z'), 1) \;&\;   (c(z'), 1)  \;&\;  (a_3(z'),0)  
\end{array}\right]
\]
\[
  U'=
 \left[\begin{array}{rrr}
     (0,2, 0) \;&\;   (c(z'), 1) \;&\;  (c(z'), 1) \\[2mm]
  (a_1(z'),0) \;&\;    (0,-2, 0) \;&\; (a_3(z'),0) \\[2mm]   
   (b(z'), 1) \;&\;   (b(z'), 1) \;&\;    (0,2, 0)
\end{array}\right].
\]
Note that $U$ is a submatrix  of $C$, while each column of $U'$ is a permutation of the corresponding column of $U$. Therefore, by replacing the block $U$ of $C$ with $U'$, we obtain a new $2^{k+2}mn\times 3$ matrix $C'$ whose columns are still permutations of $\Gamma$. 
Denote by $L'$ the list of row sums of $C'$. 
Taking into account \eqref{Gamma:alpha=1} and considering that
\begin{align*}
&\text{
 $\sum U_h = \sum U'_{j} = \pm (0,2,0)$
 has order $2^kn$ for $h=1,2,3$ and $j=1,3$, and
 }  \\
&\text{
 $\sum U'_2 = (2,0,0)$
 has order $m$,
 }  
\end{align*}
we have 
$\omega(L_{C'}) = [\,^1m, \,^{2^{k+2}mn-1}2^kn].$
The result follows by applying Lemma \ref{lem:g=3impliesall} to $C'$.\\

Case 2: $\beta=1,3$.
Let 
$C = 
\left[\begin{array}{c}
  A \\
  B
\end{array}\right]
$ 
where $A$ is the matrix defined in Theorem \ref{Gamma-2Gamma}, with $a=3\cdot 2^{k}mn$ and $b=0$, 
and $B$ is the matrix defined in Theorem \ref{2Gamma} with $a=2^{k}mn-2$ and $b=(\beta+1)/2$.
It follows that for the list $L_C$ of row sums of $C$ we have
\begin{equation}\label{Gamma:beta=1}
\omega(L_C) = [\,^{2^{k+2}mn-\beta-1}m, \,^{\beta+1}2^kn].
\end{equation}
By Remark \ref{rem:deltapermutation}, 
we can assume that the permutations $\varphi_1$ and $\varphi_2$ 
(used in Theorem \ref{Gamma-2Gamma} to define $A$) satisfy the condition $\varphi_h(z') = z' + (1,0) = (1, 2^{k-1}n)$, with $z'=(0, 2^{k-1}n)$, hence
\[a_h(z') = \varphi_h(z') + (0,1) = (1, 2^{k-1}n + 1),\] 
for $h=1$ or $2$. 

Again, by Remark \ref{rem:deltapermutation}, we can assume that the permutation
$\psi_1$ (used in Theorem \ref{2Gamma} to define $B$)
fixes $\ol{z} = (1,0)$.

Consider the following matrices:
\[
  U=
  \left[\begin{array}{c}
  A(\varphi,z')_1 \\[2mm]
  A(\varphi,z')_2 \\[2mm]
  B_{\ol{z}} \\      
\end{array}\right]
=
 \left[\begin{array}{rrr}
  (a_1(z'),0) \;&\;    (b(z'), 1) \;&\;  (c(z'), 1) \\[2mm]
   (c(z'), 1) \;&\;   (a_2(z'),0) \;&\;  (b(z'), 1)  \\[2mm] 
      (1,0,0) \;&\;      -(1,0,0) \;&\;  (0,2^kn-2, 0)     
 \end{array}\right] 
\]
\[
  U'=
 \left[\begin{array}{rrr}
       (1,0,0)  \;&\;    (b(z'), 1)   \;&\;  (b(z'), 1) \\[2mm]
    (c(z'), 1)  \;&\;      -(1,0,0)   \;&\;  (c(z), 1)  \\[2mm] 
   (a_1(z'),0)  \;&\;   (a_2(z'),0)   \;&\;  (0,2^kn-2, 0)     
 \end{array}\right]. 
\]
Note that $U$ is a submatrix  of $C$, while each column of $U'$ is a permutation of the corresponding column of $U$. Therefore, by replacing the block $U$ of $C$ with $U'$, we obtain a new $2^{k+2}mn\times 3$ matrix $C'$ whose columns are still permutations of $\Gamma$. 
Denote by $L'$ the list of row sums of $C'$. Taking into account \eqref{Gamma:beta=1} and
considering that
\begin{align*}
&\text{
 $\sum U'_h = \sum U_j= \pm (1,0,0)$
 has order $m$, for $h,j=1,2$,
}  \\
&\text{
 $\sum U_3 = (0,2^kn-2, 0)$
 has order $2^kn$, and
}  \\
&\text{
 $\sum U'_3 = (2, 0, 0)$
 has order $m$,
} 
\end{align*}
we have 
$\omega(L_{C'}) = [\,^{2^{k+2}mn-\beta} m, \,^{\beta}2^kn].$
The result follows by applying Lemma \ref{lem:g=3impliesall} to $C'$.

\subsection{The proof of Theorem \ref{row-sum exists} when $k=1$}\label{case1}
Let $\alpha$ and $\beta$ be non-negative integers such that
$\alpha+\beta=2^{k+2}mn = 8mn$. 
We first deal with the case where
$\alpha,\beta \not\in\{0,1,2,4\}$, and let $\beta_1=2mn -\alpha_1$
$\beta_2=6mn -\alpha_2$, where $\alpha_1$ and $\alpha_2$ are defined as follows:
\[
  (\alpha_1,\alpha_2) = 
  \begin{cases}
    (3, \alpha-3)        & \text{if $3\leq\alpha\leq 6mn+3$,}\\  
    (2mn-3,\alpha-2mn+3) & \text{if $6mn+4\leq \alpha\leq 8mn-3$}. \\       
  \end{cases}
\]
Clearly, $\alpha_1\geq 3$ and $\beta_1\geq 3$ are odd, hence
Theorem \ref{2Gamma} guarantees the existence of  an $RSM_\Gamma(2\Gamma, g,[\,^{\alpha_1}m,\,^{\beta_1} 2^kn ])$, say $B$.
Furthermore, $\alpha_2\neq1$ and $\beta_2\neq 1$, hence
Theorem \ref{Gamma-2Gamma} provides
an $RSM_\Gamma(\Gamma\setminus 2\Gamma, g,[\,^{\alpha_2}m,\,^{\beta_2} 2^kn ])$, say $A$.
Therefore, 
$C=\left[
\begin{array}{c}
 A \\
 B
\end{array}
\right]
$ and Lemma \ref{lem:g=3impliesall} give the desired RSM.

The case $\alpha=1$ is dealt with in Theorems \ref{k=1,beta=1} and \ref{k=1,beta=1,n=1}, while the case $\beta=1$ is proven in Theorems \ref{k=1,alpha=1} and \ref{k=1,alpha=1,n=1}.

\begin{theorem}\label{k=1,beta=1}
Let $m\geq1$ and $n\geq3$ be odd integers and let $k=1$. Then a
$RSM_\Gamma(\Gamma,g;[^{8mn-1} m,   \;^{1} 2n])$ exists.
\end{theorem}
\begin{proof} Recall that $2\Gamma = 2G\times \{0\}$ where $2G = \Z_{m}\times 2\Z_{4n}$. 

By Theorem \ref{teohall}, there exists a $\Delta_h$-permutation 
$\varphi_h$ of $2G$, where
$\Delta_h=
    \left[ 
           \,^{2mn} (1,{0})
    \right]
$, for every $h\in\{1,2,3\}$.
Let $A = \left[A(\varphi, z)\right]$ be the row block-matrix whose blocks are 
the $3\times 3$ matrices $A(\varphi, z)$ for $z\in  2G$. 
Note that $A$ is a
$6mn\times 3$ matrix whose columns are permutations of $\Gamma\setminus 2\Gamma$.
Also, letting $L_A$ be the list of row-sums of $A$, 
by Lemma \ref{A(x)rowsum} we have that 
\[
\omega(L_A) = \omega(\Delta_1 \, \cup\, \Delta_2 \, \cup\, \Delta_3) = 
\omega(\left[\,^{6mn} (1,0,0) \right])= 
\left[\,^{6mn}m\right].
\]

By Theorems \ref{special_1} and \ref{special_2},
there are $\Lambda_i$-permutation $\psi_i$ of $2G$, with $i\in\{1,2\}$, such that
\begin{enumerate}
  \item $\Lambda_1= \left[ \,^1(0,0)\right] \ \cup \
	     \left[^1\gamma\mid \gamma\in 2G, \gamma\neq (0,2n)\right]$,	
  \item $
        \Lambda_2=
          \left[ 
           \,^{2mn-6} (1,0), \,^{3} (2,0), \,^{1} (0,4),  \,^{1} (0,2n-4), \,^{1} (0,2n)
         \right] 
        $,        
  \item $\psi_1$ fixes $(0, 0)$ and
        $\overline{\gamma}=\left(-\frac{m-1}{2},  mn+1\right)$, 
  \item $\psi_2(0,0)=(0,2n), \psi_2(0,2n) = (0,2n+4)$.
\end{enumerate}
Let $B$ denote the $2^{k}mn \times 3$ matrix whose rows $B_\gamma$, indexed over $2G$, are defined as follows: 
$
B_\gamma = \left[(\gamma,0) \;\;\; (-\psi_1(\gamma),0) \;\;\; (\psi_2(\delta),0)
 \right]
$ 
 where
\[\delta=
  \begin{cases}
    (0, 2n)  & \text{if $\gamma= (0,0)$},\\
    \psi_1(\gamma) - \gamma & \text{otherwise}.
  \end{cases}
\]
Note that the columns of $B$ are permutations of $2\Gamma$. Also, 
one can check that for the list $L_B$ of row sums of $B$
we have 
\[
  L_B = \left[ 
           \,^{2mn-6} (1,0,0), \,^{3} (2,0,0), 
           \,^{1} (0,2n+4,0),  \,^{1} (0,2n-4,0), \,^{1} (0,2n,0)
         \right] 
\]
hence, $\omega(L_B) = [\,^{2mn-3}m, \,^{2}2n, \,^{1}2]$.
It follows that each column of 
$C=
\left[\begin{array}{c}
  A \\
  B \\ 
\end{array}\right]
$
is a permutation of $\Gamma$. Clearly, the list of row sums of $C$
is $L_C = L_A\,\cup\,L_B$ and $\omega(L_C) = [\,^{8mn-3}m, \,^{2}2n, \,^{1}2]$.

Now, let $\mu\in\{\pm1\}$ such that $\mu\equiv m \pmod{4}$, and take the following four elements 
$\gamma_i$ of $2G$, whose second entry depends on $\mu$:\\

$\gamma_1 =
\begin{cases}
   \left(\frac{m-1}{2}, 3n-1 \right),    & \text{if $m=1$ and $n=3$},\\
   \left(\frac{m-1}{2}, n+\mu-2 \right), & \text{otherwise},
\end{cases}
$
\\

$\gamma_2 = 
\begin{cases}
   \left(-\frac{m-1}{2}, 3n-1 \right),    & \text{if $m=1$ and $n=3$},\\
   \left(-\frac{m-1}{2}, n+\mu-2 \right), & \text{otherwise},
\end{cases}
$
\\

$
\gamma_3 = \left(-\frac{m-1}{4}, (3+\mu)n- \mu-1\right),
$\\

$
\gamma_4 = \left(\frac{m-1}{4}, (3-\mu)n-\mu-3\right).
$\\

Consider the submatrix 
$U=
\left[\begin{array}{c}
  S \\
  T \\ 
\end{array}\right]
$
of 
$C=
\left[\begin{array}{c}
  A \\
  B \\ 
\end{array}\right]
$, where $S$ and $T$ are submatrices of $A$ and $B$, respectively:
\[S=
 \left[\begin{array}{ccc}
 (a_1(\gamma_1), 0)   \;&\;  (b(\gamma_1), 1)    \;&\;  (c(\gamma_1), 1) \\[2mm]
   (c(\gamma_2), 1)   \;&\;  (a_2(\gamma_2), 0)  \;&\;  (b(\gamma_2), 1) \\[2mm] 
 (a_1(\gamma_3), 0)   \;&\;  (b(\gamma_3), 1)    \;&\;  (c(\gamma_3), 1) \\[2mm]
   (c(\gamma_4), 1)   \;&\;  (a_2(\gamma_4), 0)  \;&\;  (b(\gamma_4), 1) \\[2mm]    
 \end{array}\right], 
\]
\[
 T=
 \left[\begin{array}{ccc}
  (0,0,0) \;&\;    (0,0, 0) \;&\;  (0,2n+4,0) \\[2mm]
  (\bar{\gamma},0) \;&\;     -(\bar{\gamma},0) \;&\;  (0,2n, 0)   \\[2mm] 
 \end{array}\right].
\]
Considering that $\gamma_1\neq \gamma_3$ and $\gamma_2\neq \gamma_4$, then $S$ is well-defined, that is, $S$ contains four distinct rows of $A$. 
We denote by $C'$ the matrix obtained from $C$ by replacing $U$ with the matrix
$U'$ defined below
\[U'=
 \left[\begin{array}{ccc}
            (0,0,0)  \;&\;   (b(\gamma_1), 1)    \;&\;  (b(\gamma_2), 1) \\[2mm]
   (c(\gamma_2), 1)  \;&\;          (0,0, 0)     \;&\;  (c(\gamma_1), 1) \\[2mm] 
   (\bar{\gamma},0)      \;&\;   (b(\gamma_3), 1)    \;&\;  (b(\gamma_4), 1) \\[2mm]
   (c(\gamma_4), 1)  \;&\;        -(\bar{\gamma},0) \;&\;  (c(\gamma_3), 1) \\[2mm]    
  (a_1(\gamma_1), 0) \;&\;   (a_2(\gamma_2), 0)  \;&\;  (0,2n-2\mu+2,0) \\[2mm]
  (a_1(\gamma_3), 0) \;&\;   (a_2(\gamma_4), 0)  \;&\;  (0,2n+2\mu+2, 0)   \\[2mm] 
 \end{array}\right]. 
\]
Note that each column of $U'$ is a permutation of the corresponding column of $U$.
Therefore, each column of $C'$ is a permutation of $\Gamma$.

Considering that
\begin{align*}
&\text{
 $\sum S_h$ has order $m$, for every $h=1,\ldots, 4$,
}  \\
&\text{
 $\sum T_1 = (0,2n+4, 0)$
 has order $2n$,
}  \\
&\text{
 $\sum T_2 = (0,2n, 0)$
 has order $2$, and
}  \\
&\text{
 $\sum U'_h$ has order $m$, for every $h=1,\ldots, 6$
} 
\end{align*}
and denoting by $L_{C'}$ the list of row sums of $C'$, we have that
$\omega(L_{C'}) = [\,^{8mn-1} m, \,^{1}2n].$
The result follows by applying Lemma \ref{lem:g=3impliesall} to $C'$.
\end{proof}

\begin{theorem}\label{k=1,beta=1,n=1}
Let $m\geq1$ be an odd integer and let $k=n=1$. Then a
$RSM_\Gamma(\Gamma,g;[^{8m-1} m,   \;^{1} 2])$ exists.
\end{theorem}
\begin{proof} Recall that $2\Gamma = 2G\times \{0\}$ where $2G = \Z_{m}\times 2\Z_{4}$. 
Also, $\Gamma=\left(\Gamma\setminus2\Gamma\right)\cup 2\Gamma$.
By Theorem \ref{Gamma-2Gamma}, there exists a  $RSM_\Gamma(\Gamma\setminus2\Gamma,g;[^{6m} m])$. Therefore, it is left 
to show that a $RSM_\Gamma(2\Gamma,g;[^{2m-1} m,   \;^{1} 2])$ exists.

By Theorem \ref{teohall}
and Remark \ref{rem:deltapermutation}, there are $\Lambda_i$-permutations $\psi_i$ of $2\Gamma$, 
with $i\in\{1,2\}$, such that
\begin{enumerate}
  \item $\Lambda_1= \left[ \,^1(0,0)\right] \ \cup \
	     \left[^1z\mid z\in 2G, z\neq (0,2)\right]$, with $\psi_1(0,0)=(0,0)$, and
  \item $
        \Lambda_2=
          \left[ 
            \,^{2m-6} (1,0),   \,^{3} (2,0),   \,^{2} (0,2),   \,^{1} (0,0)
         \right]
        $.
\end{enumerate}
It is easy to see that $\psi_2$ can be chosen so that it fixes $(2,0)$ and swaps $(0,0)$ and $(0,2)$.

Denote by $z_0, z_1$ the elements of $2G$, with $z_0\neq(0,0)$, such that
\[
 \psi_1(z_0) = z_0,  \,\,\,\text{and}\,\,\,  \psi_1(z_1) - z_1 = (2,0),
\] 
and let $B$ denote the $2^{k}mn \times 3$ matrix (with entries from $2\Gamma$) 
whose rows $B_z$, indexed over $2G$, are defined as follows: 
$B_z = \left[ (z,0) \;\;\; (-\psi_1(z),0) \;\;\; (\psi_2(w),0) \right]$ where
\[w=
  \begin{cases}
    (2, 0)  & \text{if $z= {z_0}$},\\
    (0, 2)  & \text{if $z= {z_1}$},\\    
    \psi_1(z) - z & \text{otherwise}.
  \end{cases}
\]
One can check that the columns of $B$ are permutations of $2\Gamma$, and
for the list $L_B$ of row sums of $B$ we have 
\[
\omega(L_B) = [\,^{2m-1}m, \,^{1}2].
\]
The result then follows by applying Lemma \ref{lem:g=3impliesall} to $C$.
\end{proof}

\begin{theorem}\label{k=1,alpha=1}
Let $m\geq1$ and $n\geq3$ be odd integers, and let $k=1$. 
Then a
$RSM_\Gamma(\Gamma,g;[^{1} m,   \;^{8m-1} 2])$ exists.
\end{theorem}
\begin{proof} Recall that $2\Gamma = 2G\times \{0\}$ where $2G = \Z_{m}\times 2\Z_{4n}$. 

By Theorem \ref{teohall}, there exists a $\Delta_h$-permutation 
$\varphi_h$ of $2G$, where
$\Delta_h=
    \left[ 
           \,^{2mn} (0, 2)
    \right]
$, for every $h\in\{1,2,3\}$.
Let $A = \left[A(\varphi, z)\right]$ be the row block-matrix whose blocks are 
the $3\times 3$ matrices $A(\varphi, z)$ for $z\in  2G$. 
Note that $A$ is a
$6mn\times 3$ matrix whose columns are permutations of $\Gamma\setminus 2\Gamma$.
Also, letting $L_A$ be the list of row-sums of $A$, 
by Lemma \ref{A(x)rowsum} we have that 
\[
\omega(L_A) = \omega(\Delta_1 \, \cup\, \Delta_2 \, \cup\, \Delta_3) = 
\omega(\left[\,^{6mn} (0,2,0) \right]) = 
\left[\,^{6mn}2n\right].
\]

By Theorems \ref{special_1} and \ref{teohall},
there are $\Lambda_i$-permutations $\psi_i$ of $2G$, with $i\in\{1,2\}$, such that
\begin{enumerate}
  \item $\Lambda_1= \left[ \,^1(0,0)\right] \ \cup \
	     \left[^1\gamma\mid \gamma\in 2G, \gamma\neq (0,2n)\right]$,	
  \item $
        \Lambda_2=
          \left[ 
           \,^{2mn} (0,2)
         \right] 
        $,        
  \item $\psi_1$ fixes $(0, 0)$ and
        $\overline{\gamma}=\left(-\frac{m-1}{2},  mn+1\right)$.
\end{enumerate}
Let $B$ denote the $2^{k}mn \times 3$ matrix whose rows $B_\gamma$, indexed over $2G$, are defined as follows: 
$
B_\gamma = \left[(\gamma,0) \;\;\; (-\psi_1(\gamma),0) \;\;\; (\psi_2(\delta),0)
 \right]
$ 
 where
\[\delta=
  \begin{cases}
    (0, 2n)  & \text{if $\gamma= (0,0)$},\\
    \psi_1(\gamma) - \gamma & \text{otherwise}.
  \end{cases}
\]
Note that the columns of $B$ are permutations of $2\Gamma$. Also, 
one can check that for the list $L_B$ of row sums of $B$
we have 
\[
  L_B = \left[ 
           \,^{2mn-1} (0,2,0), \,^{1} (0,2n+2,0)
         \right],
\]
hence, $\omega(L_B) = [\,^{2mn-1}2n, \,^{1}u]$, where $u$ is the order of $2n+2$ in $2\Z_{4n}$.
It follows that each column of 
$C=
\left[\begin{array}{c}
  A \\
  B \\ 
\end{array}\right]
$
is a permutation of $\Gamma$. Clearly, the list of row sums of $C$
is $L_C = L_A\,\cup\,L_B$ and $\omega(L_C) = [\,^{8mn-1}2n, \,^{1}u]$.

Now, take the following four elements $\gamma_i$ of $2G$:\\

$
\gamma_1 = \left(1, 2n-2\right),
$\\

$
\gamma_2 = \left(1, 2n-6\right),
$\\

$
\gamma_3 = 
\begin{cases}
  \left(-\frac{m-1}{2}, 3n-3\right),  & 	\text{if $mn\equiv 1\pmod{4}$},\\
  \left(-\frac{m-1}{2}, 3n-5\right), & 	\text{if $mn\equiv 3\pmod{4}$},\\ 
\end{cases}  
$\\

$
\gamma_4 = 
\begin{cases}
  \left(\frac{m-1}{2}, 3n-3\right),  & 	\text{if $mn\equiv 1\pmod{4}$},\\
  \left(\frac{m-1}{2}, 3n-1\right), & 	\text{if $mn\equiv 3\pmod{4}$}.\\ 
\end{cases}  
$\\

Consider the submatrix 
$U=
\left[\begin{array}{c}
  S \\
  T \\ 
\end{array}\right]
$
of 
$C=
\left[\begin{array}{c}
  A \\
  B \\ 
\end{array}\right]
$, where $S$ and $T$ are submatrices of $A$ and $B$, respectively:
\[S=
 \left[\begin{array}{ccc}
 (a_1(\gamma_1), 0)   \;&\;  (b(\gamma_1), 1)    \;&\;  (c(\gamma_1), 1) \\[2mm]
   (c(\gamma_2), 1)   \;&\;  (a_2(\gamma_2), 0)  \;&\;  (b(\gamma_2), 1) \\[2mm] 
 (a_1(\gamma_3), 0)   \;&\;  (b(\gamma_3), 1)    \;&\;  (c(\gamma_3), 1) \\[2mm]
   (c(\gamma_4), 1)   \;&\;  (a_2(\gamma_4), 0)  \;&\;  (b(\gamma_4), 1) \\[2mm]    
 \end{array}\right], 
\]
\[
 T=
 \left[\begin{array}{ccc}
  (0,0,0) \;&\;    (0,0, 0) \;&\;  (0,2n+2,0) \\[2mm]
  (\bar{\gamma},0) \;&\;     -(\bar{\gamma},0) \;&\;  (0,2, 0)   \\[2mm] 
 \end{array}\right].
\]
Considering that $\gamma_1\neq \gamma_3$ and $\gamma_2\neq \gamma_4$, then $S$ contains four distinct rows of $A$. 
We denote by $C'$ the matrix obtained from $C$ by replacing $U$ with the matrix
$U'$ defined below
\[U'=
 \left[\begin{array}{ccc}
            (0,0,0)  \;&\;   (b(\gamma_1), 1)    \;&\;  (b(\gamma_2), 1) \\[2mm]
   (c(\gamma_2), 1)  \;&\;          (0,0, 0)     \;&\;  (c(\gamma_1), 1) \\[2mm] 
   (\bar{\gamma},0)      \;&\;   (b(\gamma_3), 1)    \;&\;  (b(\gamma_4), 1) \\[2mm]
   (c(\gamma_4), 1)  \;&\;        -(\bar{\gamma},0) \;&\;  (c(\gamma_3), 1) \\[2mm]    
  (a_1(\gamma_3), 0) \;&\;   (a_2(\gamma_4), 0)  \;&\;  (0,2n+2,0)  \\[2mm] 
  (a_1(\gamma_1), 0) \;&\;   (a_2(\gamma_2), 0)  \;&\;  (0,2, 0) \\[2mm]  
 \end{array}\right]. 
\]
Note that each column of $U'$ is a permutation of the corresponding column of $U$.
Therefore, each column of $C'$ is a permutation of $\Gamma$.

Considering that
\begin{align*}
&\text{
 $\sum S_h$ has order $2n$, for every $h=1,\ldots, 4$,
}  \\
&\text{
 $\sum T_1 = (0,2n+2, 0)$
 has order $u$,
}  \\
&\text{
 $\sum T_2 = (0,2, 0)$
 has order $2n$, and
}  \\
&\text{
 $\sum U'_h$ has order $2n$, for every $h=1,\ldots, 5$
} \\
&\text{
 $\sum U'_6$ has order $m$,
} 
\end{align*}
and denoting by $L_{C'}$ the list of row sums of $C'$, we have that
$\omega(L_{C'}) = [\,^{8mn-1} 2n, \,^{1}m].$ 
The result follows by applying Lemma \ref{lem:g=3impliesall} to $C'$.
\end{proof}

\begin{theorem}\label{k=1,alpha=1,n=1}
Let $m\geq1$ be an odd integer and let $k=n=1$. Then a
$RSM_\Gamma(\Gamma,g;[^{1} m,   \;^{8m-1} 2])$ exists.
\end{theorem}
\begin{proof} Recall that $2\Gamma = 2G\times \{0\}$ where $2G = \Z_{m}\times 2\Z_{4}$. 
Also, $C_{g}[\Gamma\setminus2\Gamma]$ and $C_{g}[2\Gamma]$ decompose $C_{g}[8m]$.

By Theorem \ref{teohall}, there exists a $\Delta_h$-permutation 
$\varphi_h$ of $2G$, where
$\Delta_h=
    \left[ 
           \,^{2mn} (0,2)
 \right]$, for every $h\in\{1,2,3\}$.
Let $A = \left[A(\varphi, z)\right]$ be the row block-matrix whose blocks are 
the $3\times 3$ matrices $A(\varphi, z)$ for $z\in  2G$. 
Note that $A$ is a
$6mn\times 3$ matrix whose columns are permutations of $\Gamma\setminus 2\Gamma$.
Also, letting $L_A$ be the list of row-sums of $A$, 
by Lemma \ref{A(x)rowsum} we have that 
\[
\omega(L_A) = \omega(\Delta_1 \, \cup\, \Delta_2 \, \cup\, \Delta_3) = 
\omega(\left[\,^{6m} (0,2,0) \right]) = 
\left[\,^{6m}2\right].
\]

By Theorem \ref{teohall}
and Remark \ref{rem:deltapermutation}, there are $\Lambda_i$-permutations $\psi_i$ of $2G$, with $i\in\{1,2\}$, such that
\begin{enumerate}
  \item $\Lambda_1= \left[ \,^1(0,0)\right] \ \cup \
	     \left[^1z\mid z\in 2G, z\neq (0,2)\right]$, with $\psi(0,0)=(0,0)$, and
  \item $
        \Lambda_2=
          \left[ 
            \,^{2m} (0,2)
         \right]
        $.
\end{enumerate}
Let $B$ denote the $2^{k}m \times 3$ matrix whose rows $B_\gamma$, indexed over $2G$, are defined as follows: 
$
B_\gamma = \left[(\gamma,0) \;\;\; (-\psi_1(\gamma),0) \;\;\; (\psi_2(\delta),0)
 \right]
$ 
 where
\[\delta=
  \begin{cases}
                     (0, 2) & \text{if $\gamma= (0,0)$},\\
    \psi_1(\gamma) - \gamma & \text{otherwise}.
  \end{cases}
\]
Note that the columns of $B$ are permutations of $2\Gamma$. Also, 
one can check that for the list $L_B$ of row sums of $B$
we have 
\[
  L_B = \left[ 
           \,^{2m-1} (0,2,0), \,^{1} (0,0,0)
         \right],
\]
hence, $\omega(L_B) = [\,^{2m-1}2, \,^{1}1]$.
It follows that each column of 
$C=
\left[\begin{array}{c}
  A \\
  B \\ 
\end{array}\right]
$
is a permutation of $\Gamma$. Clearly, the list of row sums of $C$
is $L_C = L_A\,\cup\,L_B$ and $\omega(L_C) = [\,^{8m-1}2, \,^{1}1]$.

Set $\gamma_1=(1,0)$ and $\gamma_2=(1,2)$, and note that each $\varphi_i$ swaps 
$\gamma_1$ and $\gamma_2$.
Consider the submatrix 
$U=
\left[\begin{array}{c}
  S \\
  T \\ 
\end{array}\right]
$
of 
$C=
\left[\begin{array}{c}
  A \\
  B \\ 
\end{array}\right]
$, where $S$ and $T$ are submatrices of $A$ and $B$, respectively defined as follows:
\[S=
 \left[\begin{array}{ccc}
 (a_1(\gamma_1), 0)   \;&\;  (b(\gamma_1), 1)    \;&\;  (c(\gamma_1), 1) \\[2mm]
   (c(\gamma_2), 1)   \;&\;  (a_2(\gamma_2), 0)  \;&\;  (b(\gamma_2), 1) \\[2mm] 
 \end{array}\right], 
\]
\[
 T=
 \left[\begin{array}{ccc}
  (0,0,0) \;&\;    (0,0, 0) \;&\;  (0,0, 0) \\[2mm]
 \end{array}\right].
\]
We denote by $C'$ the matrix obtained from $C$ by replacing $U$ with the matrix
$U'$ defined below
\[U'=
 \left[\begin{array}{ccc}
 (a_1(\gamma_1), 0)   \;&\; (a_2(\gamma_2), 0)     \;&\;  (0,0, 0) \\[2mm]
   (c(\gamma_2), 1)   \;&\;  (b(\gamma_1), 1)  \;&\;  (b(\gamma_2), 1) \\[2mm] 
    (0,0,0) \;&\;    (0,0, 0) \;&\;  (c(\gamma_1), 1) \\[2mm]
 \end{array}\right]. 
\]
Note that each column of $U'$ is a permutation of the corresponding column of $U$.
Therefore, each column of $C'$ is a permutation of $\Gamma$.

Considering that
\begin{align*}
&\text{
 $\sum S_h$ has order $2$, for  $h=1,2$,
}  \\
&\text{
 $\sum T = (0,0, 0)$
 has order $1$,
}  \\
&\text{
 $\sum U'_1=(2,0,0)$ has order $m$, 
} \\
&\text{
 $\sum U'_h$ has order $2$, for  $h=2, 3$
} 
\end{align*}
and denoting by $L_{C'}$ the list of row sums of $C'$, we have that
$\omega(L_{C'}) = [\,^{8mn-1} 2, \,^{1}m].$
The result follows by applying Lemma \ref{lem:g=3impliesall} to $C'$.
\end{proof}

\subsection{The proof of Theorem \ref{row-sum exists} when $k=0$}\label{case0}
Let $\alpha$ and $\beta$ be non-negative integers such that
$\alpha+\beta=4mn$. First, we assume that both $\alpha\neq1$ and $\beta\neq1$ and let 
$\beta_1=mn -\alpha_1$
$\beta_2=3mn -\alpha_2$ where $\alpha_1$ and $\alpha_2$ are defined as follows:
\begin{align*}
  \alpha_1=&
  \begin{cases}
    mn & \text{if $mn+2 \leq \alpha\leq 4mn-2$ or $\alpha=4mn$,} \\      
    \alpha-4  & \text{if $\alpha = mn+1$,}\\
    \alpha-2 & \text{if $\alpha=mn-1$}, \\   
    \alpha & \text{if $2\leq \alpha\leq mn-2$ or $\alpha\in\{0,mn\}$}, \\       
  \end{cases}
\\  
  \alpha_2 =& \alpha-\alpha_1.
\end{align*}
Since $\alpha_1\neq 1$ and $\beta_1\neq 1$, 
Theorem \ref{2Gamma} guarantees the existence of  an $RSM_\Gamma(2\Gamma, g,[\,^{\alpha_1}m,\,^{\beta_1} n ])$.
Furthermore, one can check that $\alpha_2\neq1$ and $\beta_2\neq 1$.
Hence, by Theorem \ref{Gamma-2Gamma} there is 
an $RSM_\Gamma(\Gamma\setminus 2\Gamma, g,[\,^{\alpha_2}m,\,^{\beta_2} n ])$.
Therefore, 
$C=\left[
\begin{array}{c}
 A \\
 B
\end{array}
\right]
$ and Lemma \ref{lem:g=3impliesall} provide the desired RSM.

It is then left to deal with the case $\alpha=1$; indeed, since both $m$ and $n$ are odd, the case $\beta=1$ can be obtained by exchanging the roles of $m$ and $n$.

Let $\alpha=1$.
By Theorem \ref{teohall}, there exists a $\Delta_h$-permutation 
$\varphi_h$ of $\Z_m\times 2\Z_{2n}$, with
\[\Delta_h=
    \left[ 
           \,^{(mn+1)/2} (0,{2}), \,^{(mn-3)/2} (0,-2) , (0,-4)
    \right],
\]
where $\omega(\Delta_h)=[^{mn}n]$, for every $h\in\{1,2,3,4\}$.
Now let $A = \left[A(\varphi, (x,y))\right]$ be the column block-matrix whose blocks are 
the matrices $A(\varphi, (x,y))$ for $(x,y)\in \Z_m\times 2 \Z_{2n}$, 
where $\varphi=(\varphi_1, \varphi_2, \varphi_3)$. Note that $A$ is a
$(3\cdot 2^{k}mn)\times 3$ matrix whose columns are permutations of $\Gamma\setminus 2\Gamma$.
Also, letting $L_A$ be the list of row-sums of $A$, 
by Lemma \ref{A(x)rowsum} we have that 
\[
\omega(L_A) = \omega(\Delta_1) \ \cup\ \omega(\Delta_2) \ \cup\ \omega(\Delta_3) = 
\left[\,^{3mn}n\right].
\]
Now let $B$ be the matrix whose rows are indexed over $2G$ such that
\[
B_{(x,y)}=\left[\begin{array}{rrr}(x,y,0)\;&\;(-2x,-2y,0)\;&\; (\varphi_4(x,y),0)\end{array}\right].
\] 
Notice that each column of $B$ is a permutation of $2G$ and
$\sum B(x,y)=(\varphi_4(x,y)-(x,y),0)$. Hence, letting $L_B$ be the list of row-sums of $B$ we have that
 $\omega(L_B)=\omega(\Delta_4)=\left[\,^{mn}n\right]$. Therefore,
the matrix $C = 
\left[\begin{array}{c}
  A \\
  B
\end{array}\right]
$ 
is an $RSM_\Gamma(\Gamma, 3; [\,^{4mn}n])$.

By Remark \ref{rem:deltapermutation}, 
we can assume that the permutations $\varphi_2$ and $\varphi_3$ used to define $A$  satisfy the condition $\varphi_h(1,2) = (1,0)$, hence
\[a_h(1,2) = \varphi_h(1,2) + (0,1) = (1,1),\] 
for $h=2$ or $3$. 
Furthermore, we can assume that 
\[
B_{(0,-2)}=\left[
\begin{array}{rrr}
(0,-2,0)\;&\; (0,4,0)\;&\; (0,-4,0)
\end{array}
\right].
\]
Now, consider the following matrices:
\[
  U=
  \left[\begin{array}{c}
  B_{(0,-2)} \\
  A(\varphi,(1,2))_2 \\
  A(\varphi,(1,2))_3 \\    
\end{array}\right]
=
 \left[\begin{array}{rrr}
     (0,-2, 0) \;&\;   (0,4, 0)  \;&\;  (0,-4, 0)     \\[2mm]
  (c(1,2), 1) \;&\; (a_2(1,2),0)   \;&\;  (b(1,2), 1)     \\[2mm]
    (b(1,2), 1) \;&\;   (c(1,2), 1)  \;&\;  (a_3(1,2),0)  
\end{array}\right]
\]
\begin{align*}
  U'=&
 \left[\begin{array}{rrr}
    (0,-2, 0) \;&\;   (a_2(1,2), 0) \;&\;  (a_3(1,2), 0) \\[2mm]
  (b(1,2),1) \;&\;   (0,4, 0) \;&\; (b(1,2),1) \\[2mm]   
   (c(1,2), 1) \;&\;   (c(1,2), 1) \;&\;   (0,-4, 0)
\end{array}\right]\\
=&
 \left[\begin{array}{rrr}
    (0,-2, 0) \;&\;   (1,1, 0) \;&\;  (1,1, 0) \\[2mm]
  (b(1,2),1) \;&\;   (0,4, 0) \;&\; (b(1,2),1) \\[2mm]   
   (c(1,2), 1) \;&\;   (c(1,2), 1) \;&\;   (0,-4, 0)
\end{array}\right].
\end{align*}
Notice that $\sum U'_1=(2,0,0)$ and $\sum U'_h=(0,-4,0)$ for $h=2,3$.
Note that $U$ is a submatrix  of $C$, while each column of $U'$ is a permutation of the corresponding column of $U$. Therefore, by replacing the block $U$ of $C$ with $U'$, we obtain a new $4mn\times 3$ matrix $C'$ whose columns are still permutations of $\Gamma$. 
Denoting by $L'$ the list of row sums of $C'$, 
and taking into account the values $\sum U'_h$, 
we have  that
$\omega(L_{C'}) = [\,^1m, \,^{4mn-1}n]$. Therefore, $C'$  is
an $RSM_\Gamma(\Gamma, 3; [\,^1m, \,^{4mn-1}n])$, and
the result follows by applying Lemma \ref{lem:g=3impliesall} to $C'$.

\section{The proof of Theorem \ref{main}}\label{sec:main}

\begin{theorem}\label{C_g[u]}
Let $g\geq 3$, $k\geq 0$, and let $m,n\geq 1$ be odd integers. Then HWP$(C_g[2^{k+2}mn]; gm, 2^kgn; \alpha, \beta)$ has a solution if and only if $\alpha+\beta=2^{k+2}mn$.
\end{theorem}
\begin{proof} It is a straightforward consequence of
Theorems \ref{row-sum matrix to factorizations}, \ref{ell}, and \ref{row-sum exists}.
\end{proof}

We are now ready to prove the main result of this paper.\\

\noindent
\textbf{Theorem \ref{main}.}
\emph{
Let $v$, $M$ and $N$ be integers greater than 3, and let $\ell=\mbox{lcm}(M,N)$.
then a  solution to $\mathrm{HWP}(v; M, N; \alpha, \beta)$ exists if and only if $\ell\mid v$, except possibly when
\begin{itemize}
\item $\gcd(M,N)\in \{1,2\}$;
\item $4$ does not divide $v/\ell$;
\item $v=4\ell, 8\ell$;
\item $v=16\ell$ and $\gcd(M,N)$ is odd;
\item $v=24\ell$ and $\gcd(M,N)=3$.
\end{itemize}
}
\begin{proof} 
  Let $g=\gcd(M,N)>2$. We may assume that $M=gm$ and $N = 2^kgn$, where both $m$ and $n$ are odd positive integers and $k\geq0$; hence, $\ell=\mbox{lcm}(M,N) =  2^kgmn$. By assumption, we also have that $v=4\ell s$ with $s\geq 3$, $g$ is even when $s=4$, and
  $(g,s)\neq (3,6)$.
  
  Let $(t, \epsilon)=(s/2, 2)$ if $s\geq 6$ is even, otherwise set $(t,\epsilon)=(s,1$).
  We start by factorizing  $K_v$  into two graphs $G_0$ and $G_1$
  where $G_o$ is the vertex disjoint union of $t$ copies of $K_{4\ell \epsilon}$, while
  $G_1\simeq K_t[4\ell \epsilon]$.
  Also, set $(\alpha_0, \beta_0) = (2\ell \epsilon-1,0)$ if $\alpha \geq 2\ell \epsilon-1$, otherwise, set $(\alpha_0, \beta_0) = (0, 2\ell \epsilon-1)$, and set 
  $(\alpha_1, \beta_1) = (\alpha, \beta) - (\alpha_0, \beta_0)$.

  By Theorem \ref{Liu} there is a $C_g$-factorization of $K_t[g\epsilon]$, hence there exists
  a $C_g[4\ell/g]$-factorization of $K_t[g\epsilon][4\ell/g] \simeq K_t[4\ell \epsilon]\simeq G_1$. Note that $4\ell/g = 2^{k+2}mn$, therefore by applying Theorem \ref{C_g[u]} to each component of every $C_g[4\ell/g]$-factor, we obtain a solution to HWP$(G_1; gm, 2^{k}n; \alpha_1, \beta_1)$. By adding a solution to HWP$(G_0; gm, 2^{k}n; \alpha_0, \beta_0)$ we obtain the assertion.
\end{proof}

\section*{Acknowledgments}
The authors gratefully acknowledge support from the following sources. 
A.C.\ Burgess and P.\ Danziger have received support from NSERC Discovery Grants RGPIN-2019-04328 and 
RGPIN-2022-03816 respectively. 
A. Pastine ackowledges partial support from Universidad Nacional de San Luis, Argentina, grants PROICO 03-0918 and PROIPRO 03-1720, and from ANPCyT grants PICT-2020-SERIEA-04064 and PICT-2020-SERIEA-00549.
T. Traetta has received support from GNSAGA of Istituto Nazionale di Alta Matematica.

\end{document}